\documentclass{tsp-short}

\theoremstyle{plain}
\newtheorem{lem}{Lemma}[section]
\newtheorem{thm}{Theorem}[section]
\theoremstyle{definition}
\newtheorem{defn}{Definition}[section]
\theoremstyle{remark}

\begin{document}

\title
[On local time for the solution to a white noise driven heat equation]
{On local time for the solution to a white noise driven heat equation}

\author{Izyumtseva O.L.}
\address{Institute of Mathematics of Ukrainian Academy of Sciences}
\email{olaizyumtseva@yahoo.com}

\subjclass[2000]{60G15, 60J55, 60H40}
\keywords{Local time, Gaussian process, integrators, white noise, Brownian sheet, stochastic integral}

\begin{abstract}
In this article we discuss the existence of local time for a class of Gaussian processes which appears as the solutions to some stochastic evolution equations. We show that on small intervals such processes are Gaussian integrators generated by a continuously invertible operators. This allows us to conclude that the considered processes have a local time on any finite interval with respect to spatial variable.\end{abstract}

\maketitle

\section{Introduction}

In present paper we study the local time with respect to spatial variable of solution to the following white noise driven heat equation
\begin{equation}\label{eq1}
\begin{cases}
\frac{\partial x(u,t)}{\partial t}=\frac{1}{2}\frac{\partial^2}{\partial u^2}x(u,t)+W,\\
x(u,0)=0,\ t\geq 0,\ u\in{\mathbb R}.
\end{cases}
\end{equation}
Here $W$  is a 1-dimensional Gaussian space-time white noise in $L_2({\mathbb R}\times[0;+\infty)).$ Since $W$ is a random generalized function on ${\mathbb R}\times[0;+\infty),$ the equation \eqref{eq1} is understood in the sense of generalized functions. $W$ can be considered as a centered Gaussian random measure with independent values on the disjoint subsets of ${\mathbb R}\times[0; +\infty).$ Some times it is called by Brownian sheet \cite{13}. It is known \cite{1}  that solution to \eqref{eq1} is given by
\begin{equation}\label{eq2}
x(u,t)=\int^t_0\int_{{\mathbb R}} p_{t-s}(u-v)W(dv, ds),
\end{equation}
where the stochastic integral with respect to Brownian sheet $W$ is defined as an Ito integral \cite{2} by means of the isometry
\begin{equation}\label{eq3}
\begin{split}
E\Big(\int^t_0\int_{{\mathbb R}}p_{t-s}(u-v)W(dv, ds)\Big)^2=\\
=\int^t_0\int_{{\mathbb R}}p_{t-s}(u-v)^2dvds.
\end{split}
\end{equation}
Here $p_t(u)=\frac{1}{\sqrt{2\pi t}}e^{-\frac{u^2}{2t}},\ t>0,\ u\in{\mathbb R}.$ The interest in the equation \eqref{eq1} arises since it describes the random motion of the string. For example, M.Hairer in \cite{3} considered the following model. Take $N+1$ particles with positions $x_n$ immersed in a fluid and assume that nearest-neighbours are connected by harmonic springs. If the external force $F$ is acting on the particles, then the equations of motion in the overdamped regime have the following representation
\begin{equation*}
\begin{split}
\frac{dx_0}{dt}=k(x_1-x_0)+F(x_0),\\
\frac{dx_n}{dt}=k(x_{n+1}+x_{n-1}-2x_n)+F(x_n),\ n=\overline{1, N-1},\\
\frac{dx_N}{dt}=k(x_{N-1}-x_N)+F(x_N),
\end{split}
\end{equation*}
where $k$ is a constant factor characteristic of the spring: its stiffness. This constant appears from Hooke's Law. The described model can be considered as a primitive model for a polymer chain consisting $N+1$ monomers and without self-interaction. It does not take into account the effect of the molecules of water that would randomly ``kick'' the particles. Assume that these kicks occur  randomly and independently at high rate. This effect can be modelled by increments of independent Wiener processes $w_0, \ldots, w_N$ as follows
\begin{equation*}
\begin{split}
dx_0=k(x_1-x_0)dt+F(x_0)dt+\sigma dw_0(t),\\
dx_n=k(x_{n+1}+x_{n-1}-2x_n)dt+F(x_n)d t+\sigma dw_n(t),\ n=\overline{1, N-1},\\
dx_N=k(x_{N-1}-x_N)dt+F(x_N)dt+\sigma dw_N(t).
\end{split}
\end{equation*}
Formally taking the continuum limit (with the scalings $k\approx\nu N^2$ and $\sigma\approx\sqrt{N}$) as $N\to+\infty$ one can see that this system is well-described by the solution to a stochastic partial differential equation
\begin{equation}\label{eq4}
\frac{\partial x(u,t)}{\partial t}=\nu\frac{\partial^2}{\partial u^2}x(u,t)+F(x(u,t))+W.
\end{equation}

More general $a$ white noise driven heat equation was considered by T.Funaki in \cite{4}. For given two functions $a: {\mathbb R}^d\to{\mathbb R}^{d\times d}$ and $b: {\mathbb R}^d\to{\mathbb R}^d,$ let $x$ be a diffusion process on ${\mathbb R}^d$ determined by the stochastic differential equation
\begin{equation}\label{eq5}
dx(t)=a(x(t))dW(t)+b(x(t))dt,
\end{equation}
where $W$ is a $d$-dimensional Wiener process. Funaki introduced the following $C([0;1], {\mathbb R}^d)$-valued stochastic differential equation
\begin{equation}\label{eq6}
\frac{dX_t(s)}{dt}=a(X_t(s))W+b(X_t(s))+\frac{k}{2}\frac{\partial^2}{\partial s^2}X_t(s),
\end{equation}
$s\in[0;1].$ Here $W$ is a $d$-dimensional Gaussian white noise with two parameters in $L_2([0;+\infty]\times[0;1]).$ The equation \eqref{eq6} describes a string which moves in ${\mathbb R}^d$ being interfered by the process $x.$ The idea behind the derivation of the equation \eqref{eq6} is to take the scaling limit of a sequence of polygonal approximations. More precisely, Funaki approximates the string by a polygon and sets on each corner a particle which moves governed by a stochastic differential equation \eqref{eq5} with a suitable scaling. In each step of the approximation, the interaction between neighboring two particles was always taken into account. M.Kardar, G.Parisi and Y.-C.Zhang in \cite{5} proposed a model for the time evolution of the profile of a growing interface. The interface profile is described by a height $h(u,t).$ For convenience overhangs are ignored. The Langevin equation for a local growth of the profile is given by
\begin{equation}
\label{eq7}
\frac{\partial h}{\partial t}=\nu\nabla^2h+\frac{\lambda}{2}(\nabla h)^2+W.
\end{equation}

The mentioned equations \eqref{eq1},\ \eqref{eq4},\ \eqref{eq7} can be considered as the mathematical models which describe the motion of curve at the random medium. Since all these equations constructed by means of white noise, then the corresponding solutions would be nonsmooth random processes. As it was discussed in works of authors \cite{6,7} geometric characteristics of random processes are local times and self-intersection local times. At present paper we want to examine the existence of local time for the solution to the simplest of these equations. To do this we will show that the solution to the  equation \eqref{eq1} belongs to the class of Gaussian integrators. This class of Gaussian processes firstly was introduced by A.A.Dorogovtsev in the work \cite{8}. The original definition is the following.
\begin{defn}\label{defn1.1} \cite{8} A centered Gaussian process $x(t),\  t\in[0;1],\ x(0)=0$ is said to be an integrator if there exists the constant $c>0$ such that for an arbitrary partition $0=t_0<t_1<\ldots<t_n=1$ and real numbers $a_0, \ldots, a_{n-1}$ the following relation holds
$$
E\Big(\sum^{n-1}_{k=0}a_k(x(t_{k+1})-x(t_k))\Big)^2\leq c\sum^{n-1}_{k=0}a^2_k\Delta t_k.
$$
This inequality allows to define a stochastic integral for any square integrable function with respect to the integrator. The following statement describes the structure of all Gaussian integrators.
\end{defn}
\begin{lem}\label{lem1.1}\cite{8}
The centered Gaussian process $x(t),\ t\in[0;1]$ is the integrator iff there exist a white noise $\xi$ in the space $L_2([0;1])$and a continuous linear operator in the same space such that
$$
x(t)=(A{1\!\!\,{\rm I}}_{[0;t]}, \xi),\ t\in[0;1].
$$
\end{lem}
It follows from Lemma \ref{lem1.1} that all properties of the Gaussian integrator can be characterized in terms of the operator $A.$

Since we are going to study the local time for  the solution to the equation \eqref{eq1}, let us recall the general definition of local time for a 1-dimensional random process $y(t),\ t\in[T_1, T_2].$
\begin{defn}\label{defn1.2} For any $p\in{\mathbb N},\ z\in{\mathbb R}$
$$
L_p\mbox{-}\lim_{\varepsilon\to0}\int^{T_2}_{T_1}p_\varepsilon(y(s)-z)ds=\int^{T_2}_{T_1}\delta_{z}(y(s))ds=:l^{y}(z)
$$
is said to be a local time of the process $y$ at the point $z$ whenever the limit exists.
\end{defn}
Put $l^{y}:=l^{y}(0).$

This paper is organized as follows. In the beginning we prove that solution to a stochastic heat equation at the fixed moment of time is the integrator. Then we check that for a sufficiently small interval this Gaussian integrator is generated by a continuously invertible linear operator. Using that, we establish the existence of local time with respect to spatial variable on any finite interval.
\section{The existence of local time}
In this section we will prove that the solution $x$ to the equation \eqref{eq1} considered at the fixed moment of the time has a local time with respect to spatial variable. Our proof of existence of the local time for $x$ has the following outline. First of all we prove that $x$ belongs to the class of Gaussian integrators. The local times of integrators were discussed in the works \cite{6,7,9,14}. It was proved that if the integrator is generated by a continuously invertible operator, then it has a local time. We will check that on the small intervals the process $x$ satisfies this condition. Then we discuss the existence of the local time for a linear combination of Gaussian process and some jointly Gaussian random variable and finally prove the main statement for an arbitrary finite interval.

As it is mentioned in Introduction the solution to \eqref{eq1} for $t=1$ has a form
$$
x(u)=\int^1_0\int_{{\mathbb R}}p_{1-s}(u-v)W(dv, ds),\ u\in[U_1; U],
$$
where $W$ is the Brownian sheet in the space $L_2({\mathbb R}\times[0;1]).$
Denote by
$$
x_{U_1}(u)=x(u)-x(U_1), \ u\in[U_1; U_2].
$$
\begin{thm}\label{thm2.1}
The process $x_{U_1}(u),\ u\in[U_1, U_2]$ is a Gaussian integrator.
\end{thm}
\begin{proof}
Denote by
$$
\widehat{f}(\lambda)=\frac{1}{\sqrt{2\pi}}\int_{{\mathbb R}}f(y)e^{-i\lambda y}dy.
$$
Using the properties of stochastic integral and isometry property of Fourier transform one can get the following relations
$$
E\Big(\sum^{n-1}_{k=0}a_k(x_{U_1}(u_{k+1})-x_{U_1}(u_{k}))\Big)^2=
$$
$$
=
E\Big(\sum^{n-1}_{k=0}a_k(x(u_{k+1})-x(u_{k}))\Big)^2=
$$
$$
=
E\Big(\sum^{n-1}_{k=0}a_k\Big(
\int^1_0\int_{{\mathbb R}}(p_{1-s}(u_{k+1}-v)-p_{1-s}(u_k-v))W(dv,ds)
\Big)\Big)^2=
$$
$$
=\int^1_0\int_{{\mathbb R}}
\Big(\sum^{n-1}_{k=0}a_k(p_{1-s}(u_{k+1}-v)-p_{1-s}(u_k-v))
\Big)\Big)^2dvds=
$$
$$
=\int^1_0\int_{{\mathbb R}}
\Big|\sum^{n-1}_{k=0}a_k(\widehat{p}_{1-s}(u_{k+1},\lambda)-\widehat{p}_{1-s}(u_k,\lambda))
\Big|^2d\lambda ds=
$$
$$
=
\frac{1}{2\pi}
\int^1_0\int_{{\mathbb R}}
\Big|\sum^{n-1}_{k=0}a_k
e^{-\frac{(1-s)\lambda^2}{2}}
[e^{-i\lambda u_{k+1}}-e^{-i\lambda u_{k}}]
\Big|^2d\lambda ds=
$$
\begin{equation}
\label{eq9}
=
\int^1_0\int_{{\mathbb R}}\lambda^2e^{-(1-s)\lambda^2}|\widehat{f}(\lambda)|^2d\lambda ds,
\end{equation}
where
$$
f(y)=\sum^{n-1}_{k=0}a_k{1\!\!\,{\rm I}}_{[u_k; u_{k+1}]}(y).
$$
By integrating \eqref{eq9} over $s$ one can obtain that \eqref{eq9} equals
$$
\int_{{\mathbb R}}|\widehat{f}(\lambda)|^2(1-e^{-\lambda^2})d\lambda\leq
\int_{{\mathbb R}}|\widehat{f}(\lambda)|^2d\lambda=
$$
$$
=
\int_{{\mathbb R}}f(u)^2du=\sum^{n-1}_{k=0}a^2_k\Delta u_k,
$$
which proves the theorem.
\end{proof}
 Since $x_{U_1}$ is an integrator, then there exist a continuous linear operator $A$ in the space $L_2([U_1; U_2])$ and a white noise $\xi$ in the same space such that the process $x_{U_1}(u),\ u\in[U_1; U_2]$ admits the representation $x_{U_1}(u)=\Big(A{1\!\!\,{\rm I}}_{[U_1; u]}, \xi\Big).$

Moreover, for any $f\in L_2([U_1; U_2]) $
$$
\|Af\|^2=\int_{{\mathbb R}}|\widehat{f}(\lambda)|^2(1-e^{-\lambda^2})d\lambda.
$$
Let us check that in the case of sufficiently small interval $[U_1;U_2]$ the operator $A$ is continuously invertible.
\begin{lem}\label{lem2.1}
For an arbitrary interval $[U_1; U_2]$ with $U_2-U_1<2\sqrt{\pi}$ the operator $A$ corresponding to the integrator $x_{U_1}$ is invertible.
\end{lem}
\begin{proof}
Suppose that $b-a<2\sqrt{\pi}.$ Let us check that there exists a constant $c_1>0$ such that for any $f\in L_2([a;b])$
$$
\|Af\|^2\geq c_1\|f\|^2.
$$
Really, for $f$ with ${\mathop{\rm supp}} f\in [a; b]$
$$
\|Af\|^2=\int_{{\mathbb R}}|\widehat{f}(\lambda)|^2(1-e^{-\lambda^2})d\lambda=
$$
$$
=\|f\|^2-\int_{{\mathbb R}}|\widehat{f}(\lambda)|^2e^{-\lambda^2}d\lambda.
$$
Note that
$$
\int_{{\mathbb R}}|\widehat{f}(\lambda)|^2e^{-\lambda^2}d\lambda=
\int_{{\mathbb R}}|\widehat{f}(\lambda)e^{-\frac{\lambda^2}{2}}|^2d\lambda=
$$
$$
=
\int_{{\mathbb R}}|\widehat{f*p_1}(\lambda)|^2d\lambda=\int_{{\mathbb R}}f*p_1(u)^2du=
$$
$$
=
\int_{{\mathbb R}}\Big(\int^b_af(v)\frac{1}{\sqrt{2\pi}}e^{-\frac{(u-v)^2}{2}}dv\Big)^2du.
$$
Applying the Cauchy inequality one can obtain the following relation
$$
\Big(\int^b_af(v)\frac{1}{\sqrt{2\pi}}e^{-\frac{(u-v)^2}{2}}dv\Big)^2\leq
$$
\begin{equation}\label{eq10}
\leq \|f\|^2\int^b_a\frac{1}{2\pi}e^{-(u-v)^2}dv.
\end{equation}
It follows from \eqref{eq10} that
$$
\|Af\|^2\geq\|f\|^2-\|f\|^2\int_{{\mathbb R}}\int^b_a\frac{1}{2\pi}e^{-(u-v)^2}dvdu=\Big(1-\frac{b-a}{2\sqrt{\pi}}\Big)\|f\|^2.
$$
\end{proof}
In the  \cite{9} we prove that the Gaussian integrator generated by a continuously invertible operator has a local time at any real point which is jointly continuous in time and space variables. Consequently, the Gaussian integrator $x_a(u),\ u\in[a;b]$ with $b-a<2\sqrt{\pi}$ has a jointly continuous local time at any real point. Now we will consider the interval of an arbitrary length. Consider the Gaussian integrator $x_{U_1}(u),\ u\in [U_1; U_2]$ and the partition $U_1=u_0<u_1<\ldots<u_n=U_2$ such that for any $k=\overline{0, n-1}$
$$
u_{k+1}-u_k<2\sqrt{\pi}.
$$
Every Gaussian integrator $X_{u_k}(u),\ u\in[u_k; u_{k+1}]$ has a jointly continuous local time at any real point. Does the same hold for the Gaussian integrator $X_{U_1}(u),\ u\in[U_1; U_2]?$ The answer is given by the next theorem.

\begin{thm}
\label{thm2.2}
For any $p\in{\mathbb N},\ z\in{\mathbb R}$ there exists
$$
L_p\mbox{-}\lim_{\varepsilon\to0}\int^{U_2}_{U_1}p_\varepsilon(x_{U_1}(u)-z)du=\int^{U_2}_{U_1}\delta_z(x_{U_1}(u))du.
$$
\end{thm}
\begin{proof}
Consider the partition $U_1=u_0<u_1<\ldots<u_n=U_2$ such that for any $k=\overline{0, n-1}$
$$
u_{k+1}-u_k<2\sqrt{\pi}.
$$
Then one can written the following relation
$$
\int^{U_2}_{U_1}p_\varepsilon(x_{U_1}(u)-z)du=\sum^{n-1}_{k=0}\int^{u_{k+1}}_{u_k}
p_\varepsilon(x_{u_k}(u)-(x(U_1)-x(u_k)+z))du.
$$
Denote by
$$
V_\varepsilon=\int^{u_{k+1}}_{u_k}
p_\varepsilon(x_{u_k}(u)-(x(U_1)-x(u_k)+z))du.
$$
Then to prove the theorem it suffices to check that
\begin{equation}
\label{eq11}
E(V_{\varepsilon_1}-V_{\varepsilon_2})^{2p}\to0, \ \varepsilon_1, \varepsilon_2\to0.
\end{equation}
Since
$$
E(V_{\varepsilon_1}-V_{\varepsilon_2})^{2p}=\sum^{2p}_{l=0}(-1)^{2p-l}C^l_{2p}
E(V_{\varepsilon_1})^l(V_{\varepsilon_2})^{2p-l},
$$
then \eqref{eq11} follows from the existence of finite limit $E(V_{\varepsilon_1})^l(V_{\varepsilon_2})^{2p-l}$ as $\varepsilon_1, \varepsilon_2\to0.$ Let us check it.

One can see that
$$
E(V_{\varepsilon_1})^l(V_{\varepsilon_2})^{2p-l}=
$$
$$
=E\int^{u_{k+1}}_{u_k}\ldots^{2p}\int^{u_{k+1}}_{u_k}\prod^l_{i=1}p_{\varepsilon_1}(x_{u_k}(v_i)-(x(U_1)-x(u_k)+z))\cdot
$$
\begin{equation}\label{eq12}
\cdot
\prod^{2p}_{j=l+1}p_{\varepsilon_2}(x_{u_k}(v_j)-(x(U_1)-x(u_k)+z))d\vec{v}.
\end{equation}
Since $x_{u_k}(u),\ u\in[u_k; u_{k+1}]$ is the Gaussian integrator, then there exist a continuous linear operator $\widetilde{A}$  in the space $L_2([u_k; u_{k+1}])$ and a white noise $\widetilde{\xi}$ in $L_2([u_k; u_{k+1}])$ such that
$$
x_{u_k}(u)=(\widetilde{A}{1\!\!\,{\rm I}}_{[u_k; u]}, \widetilde{\xi}).
$$
Note that the random variable $x(U_1)-x(u_k)+z$ is jointly Gaussian with the white noise $\widetilde{\xi}.$
To check this let us recall the construction of the white noise $\widetilde{\xi}.$ Denote by $\overline{LS\{x_{u_k}\}}$ the closure of the linear span of values of $x_{u_k}$ with respect to square mean norm. $\overline{LS\{x_{u_k}\}}$ is a separable Hilbert space. There exists the subspace $H_1$ of $L_2([u_k; u_{k+1}])$ which is isomorphic to $\overline{LS\{x_{u_k}\}}.$ Denote by $H_1=j(\overline{LS\{x_{u_k}\}}),$ where $j: \overline{LS\{x_{u_k}\}}\to H_1$ is isomorphism. Then $L_2([u_k; u_{k+1}])$ can be represented as a direct sum $H_1\oplus H_2.$ Suppose that $\widetilde{\xi}_2$ is a Gaussian white noise in $H_2$ which is independent of $x.$ Put $(h_1, \widetilde{\xi}_1)=j^{-1}(h_1)$ and
$$
(h_1\widetilde{\xi})=(h_1,\widetilde{\xi}_1)+(h_2,\widetilde{\xi}_2)=j^{-1}(h_1)+(h_2,\widetilde{\xi}_2).
$$
Notice that $j^{-1}(h_1)\in\overline{LS\{x_{u_k}\}}.$ Since $x(u),\ u\in[U_1; U_2]$ is the Gaussian process, then the random variables $j^{-1}(h_1)$ and $x(U_1)-x(u_k)+z$ are jointly Gaussian. The random variables $j^{-1}(h_1)$ and $(h_2, \widetilde{\xi}_2)$ are independent. It implies that the random variables $(h, \widetilde{\xi})$ and $x(U_1)-x(u_k)+z$ are jointly Gaussian. Hence, the random variable $x(U_1)-x(u_k)+z$ and $\widetilde{\xi}$ are jointly Gaussian. To continue the proof of the theorem we need the following statement which can be easily proved.
\begin{lem}\label{lem2.2}
 Let $\xi$ be a white noise in the space $L_2([0;1]).$ For a jointly Gaussian with $\xi$ random variable $\eta$ there exist $h\in L_2([0;1])$ and the random variable $\zeta,$ that is independent of the white noise $\xi$, such that $\eta$ admits the representation
 $$
 \eta=(h,\xi)+\zeta.
 $$
  \end{lem}
 Applying Lemma 2.2 to the white noise $\widetilde{\xi}$ and the random variable $x(U_1)-x(u_k)+z$ one can see that \eqref{eq12} can be represented as follows
 $$
 (2p)!
 \int_{\Delta_{2p}(u_k; u_{k+1})}
 E\prod^l_{i=1}
 p_{\varepsilon_1}((\widetilde{A}{1\!\!\,{\rm I}}_{[u_k;v_i]}, \widetilde{\xi})-(h,\widetilde{\xi})-\zeta)\cdot
 $$
 \begin{equation}
 \label{eq13}
 \cdot\prod^{2p}_{j=l+1}
 p_{\varepsilon_2}((\widetilde{A}{1\!\!\,{\rm I}}_{[u_k;v_j]}, \widetilde{\xi})-(h,\widetilde{\xi})-\zeta)d\vec{v},
 \end{equation}
 where
 $$
 \Delta_m(u_k; u_{k+1})=\{u_k\leq v_1<\ldots<v_m\leq u_{k+1}\}, \ \Delta_m:=\Delta_m(0;1).
 $$
 Denote by $P_h$ the projection onto linear span generated by $h.$ Then \eqref{eq13} possesses the representation
 $$
 (2p)!
 \int_{\Delta_{2p}(u_k; u_{k+1})}
 EE
\Big\{\Big(
\prod^l_{i=1}
 p_{\varepsilon_1}((
(I-P_h+P_h)
\widetilde{A}{1\!\!\,{\rm I}}_{[u_k;v_i]}, \widetilde{\xi})
-(h,\widetilde{\xi})-\zeta\Big)\cdot
 $$
 \begin{equation}
 \label{eq14}
\cdot \prod^{2p}_{j=l+1}
 p_{\varepsilon_2}((
(I-P_h+P_h)
\widetilde{A}{1\!\!\,{\rm I}}_{[u_k;v_j]}, \widetilde{\xi})
-(h,\widetilde{\xi})
-\zeta)/
(P_h\widetilde{A}{1\!\!\,{\rm I}}_{[u_k; v_i]}, \widetilde{\xi}),\  i=\overline{1, 2p},\ (h,\widetilde{\xi}),\ \zeta\Big\}
d\vec{v}.
 \end{equation}
 Denote by $I(\varepsilon_1, \varepsilon_2)={\mathop{\rm diag}}(\underbrace{\varepsilon_1, \ldots,\varepsilon_1}_l, \underbrace{\varepsilon_2, \ldots,\varepsilon_2}_{2p-l}).$ Let $B(g_1, \ldots, g_n)$ be the Gramian matrix constructed by elements $g_1, \ldots, g_n,$
 $$
 G(g_1, \ldots, g_n):=\det B(g_1, \ldots,g_n).
 $$
 Put
 $$
 C_{\varepsilon_1,\varepsilon_2}(\vec{v})=
 $$
 $$
 =B((I-P_h)\widetilde{A}{1\!\!\,{\rm I}}_{[u_k; v_1]},\ldots,(I-P_h)\widetilde{A}{1\!\!\,{\rm I}}_{[u_k; v_{2p}]})+I(\varepsilon_1, \varepsilon_2),
 $$
 $$
 \vec{\theta}=\begin{pmatrix}
P_h\widetilde{A}{1\!\!\,{\rm I}}_{[u_k; v_1]}-(h, \widetilde{\xi})-\zeta\\
\vdots\\
P_h\widetilde{A}{1\!\!\,{\rm I}}_{[u_k; v_{2p}]}-(h, \widetilde{\xi})-\zeta
\end{pmatrix}
$$
Then \eqref{eq14} equals
\begin{equation}
\label{eq15}
(2p)!
\int_{\Delta_{2p}(u_k;u_{k+1})}
\frac{1}
{(2\pi)\sqrt{\det C_{\varepsilon_1, \varepsilon_2}(\vec{v})}}
e^{-\frac{1}{2}(C^{-1}_{\varepsilon_1, \varepsilon_2}(\vec{v})\vec{\theta}, \vec{\theta})}d\vec{v}.
\end{equation}
It follows from \eqref{eq15} that to end the proof of the theorem one must check that the integral
\begin{equation}
\label{eq16}
\int_{\Delta_{2p}(u_k;u_{k+1})}\frac{d\vec{v}}
{\sqrt{G((I-P_h)\widetilde{A}{1\!\!\,{\rm I}}_{[u_k; v_1]}, \ldots, (I-P_h)\widetilde{A}{1\!\!\,{\rm I}}_{[u_k; v_{2p}]})}}
\end{equation}
converges. Really, for any $\vec{v}\in \Delta_{2p}(u_k; u_{k+1})$
$$
\frac{1}{\sqrt{\det C_{\varepsilon_1,\varepsilon_2}(\vec{u})}}
e^{-\frac{1}{2}(C^{-1}_{\varepsilon_1,\varepsilon_2}(\vec{v})\vec{\theta}, \vec{\theta})}\to
$$
$$
\frac{1}
{\sqrt{G((I-P_h))\widetilde{A}{1\!\!\,{\rm I}}_{[u_k; v_1]}, \ldots, (I-P_h))\widetilde{A}{1\!\!\,{\rm I}}_{[u_k; v_{2p}]}}}\cdot
$$
$$
\cdot
e^{-\frac{1}{2}(B^{-1}((I-P_h))\widetilde{A}{1\!\!\,{\rm I}}_{[u_k; v_1]}, \ldots, (I-P_h))\widetilde{A}{1\!\!\,{\rm I}}_{[u_k; v_{2p}]})\vec{\theta},\vec{\theta})}
$$
as $\varepsilon_1, \varepsilon_2\to0.$ Since for any $\varepsilon_1, \varepsilon_2>0 $ less or equal to
$$
\frac{1}
{\sqrt{G((I-P_h))\widetilde{A}{1\!\!\,{\rm I}}_{[u_k; v_1]}, \ldots, (I-P_h))\widetilde{A}{1\!\!\,{\rm I}}_{[u_k; v_{2p}]}}},
$$
then using \eqref{eq16} and applying the Lebesgue dominated convergence theorem one can obtain the statement of the theorem. Therefore, let us check \eqref{eq16}. To do this we need the following statements which are related to the properties of the Gram determinant in the integrand. Those statements were proved in the works \cite{10, 11}. Let $L$ be the finite-dimensional subspace of the space $L_2([0;1]).$ Suppose that $e_1, \ldots, e_n$ is an orthonormal basis for $L.$ Denote by $P_L$ the projection onto $L.$ One can prove the following statement.

\begin{lem}\label{lem2.2}\cite{10}
For any $g_1, \ldots, g_k\in L_2([0;1])$ the following relation holds
$$
G((I-P_L)g_1, \ldots, (I-P_L)g_k)=G(g_1, \ldots, g_k, e_1, \ldots, e_n).
$$
\end{lem}
It follows from Lemma \ref{lem2.2} that
\begin{equation}
\label{eq17}
G((I-P_h)\widetilde{A}{1\!\!\,{\rm I}}_{[u_k;v_1]}, \ldots,(I-P_h)\widetilde{A}{1\!\!\,{\rm I}}_{[u_k;v_p]} )=
G\Big(\widetilde{A}{1\!\!\,{\rm I}}_{[u_k;v_1]}, \ldots,\widetilde{A}{1\!\!\,{\rm I}}_{[u_k;v_p]}, \frac{h}{\|h\|} \Big).
\end{equation}
Since $\widetilde{A}$ is a continuously invertible operator in the space $L_2([u_k; u_{k+1}]),$ then (16) can be represented as follows
\begin{equation}
\label{eq18}
G\Big(\widetilde{A}{1\!\!\,{\rm I}}_{[u_k;v_1]}, \ldots,\widetilde{A}{1\!\!\,{\rm I}}_{[u_k;v_p]}, \widetilde{A}\widetilde{A}^{-1}\frac{h}{\|h\|} \Big).
\end{equation}
One can check that the following theorem holds.
\begin{thm}
\label{thm2.3}\cite{11}
Suppose that $A$ is a continuously invertible operator in the Hilbert space $H.$ Then for any elements $e_1, \ldots, e_n$ of the space $H$ the following relation holds
$$
G(Ae_1, \ldots, Ae_n)\geq\frac{1}{\|A^{-1}\|^{2n}}G(e_1, \ldots, e_n).
$$
\end{thm}

It follows from Theorem \ref{thm2.3} and properties of the determinant that \eqref{eq18} greater or equal to
$$
\frac{1}
{\|h\|^2\|\widetilde{A}^{-1}\|^{2p}}G({1\!\!\,{\rm I}}_{[u_k, v_1]}, \ldots, {1\!\!\,{\rm I}}_{[u_k, v_p]}, \widetilde{A}^{-1}h).
$$
Denote by $M$ the subspace of all step functions in $L.$ Let $f_1, \ldots, f_s$ be an orthonormal basis for $M$ and $e_1, \ldots, e_m $ be an orthonormal basis for the orthogonal complement of $M$ in $L.$ Note that $f_1, \ldots, f_s, e_1, \ldots, e_m$ is an orthonormal basis for $L.$ One can see that the following estimates take place.

\begin{lem}
\label{lem2.3}\cite{10}
There exists a constant $c>0$ such that the following relation holds
$$
G({1\!\!\,{\rm I}}_{[0;t_1]}, \ldots, {1\!\!\,{\rm I}}_{[0, t_k]}, f_1, \ldots, f_s, e_1, \ldots, e_m)\geq
c\ G({1\!\!\,{\rm I}}_{[0;t_1]}, \ldots, {1\!\!\,{\rm I}}_{[0, t_k]}, f_1, \ldots, f_s).
$$
\end{lem}
\begin{lem}\cite{10}
\label{lem2.4}
Let $0\leq s_1<\ldots<s_N\leq1$ be the points of jumps of functions $f_1, \ldots, f_s.$ Then there exists a constant $c_{\vec{s}}>0$ which depends on $\vec{s}=(s_1, \ldots, s_N)$ such that the following relation holds
$$
G({1\!\!\,{\rm I}}_{[0;t_1]}, \ldots, {1\!\!\,{\rm I}}_{[0, t_k]}, f_1, \ldots, f_s)\geq
c_{\vec{s}}\ G({1\!\!\,{\rm I}}_{[0;t_1]}, \ldots, {1\!\!\,{\rm I}}_{[0, t_k]}, {1\!\!\,{\rm I}}_{[0;s_1]}, \ldots, {1\!\!\,{\rm I}}_{[0, s_N]}).
$$
\end{lem}

If $\widetilde{A}^{-1}h$ is not a step function, then it follows from Lemma \ref{lem2.3} that
\begin{equation}
\label{eq19}
G({1\!\!\,{\rm I}}_{[u_k; v_1]}, \ldots,{1\!\!\,{\rm I}}_{[u_k; v_p]}, \widetilde{A}^{-1}h)\geq
c\ G({1\!\!\,{\rm I}}_{[u_k; v_1]}, \ldots,{1\!\!\,{\rm I}}_{[u_k; v_p]}).
\end{equation}
The relation \eqref{eq19} implies that \eqref{eq16} less or equal to
$$
\frac{1}
{\sqrt{c}\|h\|\|\widetilde{A}^{-1}\|^{2p}}
\int_{\Delta_{2p}(u_k; u_{k+1})}
\frac{d\vec{v}}
{\sqrt{G({1\!\!\,{\rm I}}_{[u_k; v_1]},\ldots,{1\!\!\,{\rm I}}_{[u_k; v_{2p}]} )}}=
$$
$$
=
\frac{1}
{\sqrt{c}\|h\|\|\widetilde{A}^{-1}\|^{2p}}
\int_{\Delta_{2p}(u_k; u_{k+1})}
\frac{d\vec{v}}
{\sqrt{(v_1-u_k)(v_2-v_1)\ldots(v_{2p}-v_{2p-1})}}=
$$
$$
=
\frac{\widetilde{c}(\vec{u},p)}
{\|h\|\|\widetilde{A}^{-1}\|^{2p}}
\int_{\Delta_{2p}}
\frac{d\vec{v}}
{\sqrt{v_1(v_2-v_1)\ldots(v_p-v_{p-1})}}<+\infty.
$$
Here $\widetilde{c}(\vec{u}, p)$ is positive constant which depends on $u_k, u_{k+1}$ and $p.$

If $\widetilde{A}^{-1}h$ is a step function, then it follows from Lemma \ref{lem2.4} that
\begin{equation}
\label{eq20}
G({1\!\!\,{\rm I}}_{[u_k; v_1]}, \ldots,{1\!\!\,{\rm I}}_{[u_k; v_{2p}]}, \widetilde{A}^{-1}h)\geq
c_{\vec{s}}\ G({1\!\!\,{\rm I}}_{[u_k; v_1]}, \ldots,{1\!\!\,{\rm I}}_{[u_k; v_{2p}]},{1\!\!\,{\rm I}}_{[u_k; s_1]}, \ldots,{1\!\!\,{\rm I}}_{[u_k; s_N]}),
\end{equation}
where $s_1, \ldots, s_N$ are the points of jumps of the function $\widetilde{A}^{-1}h.$ The relation \eqref{eq20} implies that \eqref{eq16} less or equal to
$$
\frac{1}{\sqrt{c_{\vec{s}}}\|h\|\|\widetilde{A}^{-1}\|^{2p}}
\int_{\Delta_{2p}(u_k; u_{k+1})}
\frac{d\vec{u}}
{\sqrt{G({1\!\!\,{\rm I}}_{[u_k; v_1]}, \ldots,{1\!\!\,{\rm I}}_{[u_k; v_{2p}]},{1\!\!\,{\rm I}}_{[u_k; s_1]}, \ldots,{1\!\!\,{\rm I}}_{[u_k; s_N]})}}.
$$
Consider the following partition of the domain $\Delta_{2p}(u_k; u_{k+1})$
$$
\Delta_{2p}(u_k; u_{k+1})=\mathop{\cup}\limits_{p=n_0+\ldots+n_N}I_{n_0\ldots n_N }(u_k; u_{k+1}),
$$
where
$$
I_{n_0\ldots n_N}(u_k; u_{k+1})=
$$
$$
=\{u_k\leq v_1\leq\ldots\leq v_{n_0}\leq s_1\leq v_{n_0+1}\leq \ldots\leq v_{n_0+n_1}\leq s_2\leq
$$
$$
\ldots\leq s_n\leq v_{n_0+\ldots+n_{N-1}\leq\ldots\leq v_p\leq u_{k+1}}\}.
$$
Note that
\begin{equation}
\label{eq21}
I_{n_0\ldots n_N}(u_k; u_{k+1})=\Delta_{n_0}(u_k; s_1)\times \Delta_{n_1}(s_1; s_2)\times\ldots\times \Delta_{n_N}(s_N; u_{k+1}).
\end{equation}
It follows from \eqref{eq21} that to finish the proof it suffices to check that
\begin{equation}
\label{eq22}
\int_{\Delta_k(s_1; s_2)}
\frac{d\vec{v}}
{\sqrt{(v_1-s_1)(v_2-v_1)\ldots(v_k-v_{k-1})(s_2-v_k)}}<+\infty.
\end{equation}
The convergence of the integral \eqref{eq22} we proved in \cite{11}. For convenience of the reader, let us now briefly recall essential points of the proof. The change of variables in the integral implies that
$$
\int_{\Delta_k(s_1; s_2)}
\frac{d\vec{v}}
{\sqrt{(v_1-s_1)(v_2-v_1)\ldots(v_k-v_{k-1})(s_2-v_k)}}=
$$
$$
=
c(\vec{s}, k)
\int_{\Delta_k}
\frac{d\vec{v}}
{\sqrt{v_1(v_2-v_1)\ldots(v_k-v_{k-1})(1-v_k)}}=
$$
$$
=
c(\vec{s}, k)
\int_{\Delta_k}
\frac{d\vec{v}}
{\sqrt{G({1\!\!\,{\rm I}}_{[0;v_1]}, \ldots,{1\!\!\,{\rm I}}_{[0;v_k]},{1\!\!\,{\rm I}}_{[0;1]})}}=
$$
\begin{equation}
\label{eq23}
=
c(\vec{s}, k)
\int_{\Delta_k}
\frac{d\vec{v}}
{\sqrt{G((I-\widetilde{P}){1\!\!\,{\rm I}}_{[0;v_1]}, \ldots,(I-\widetilde{P}){1\!\!\,{\rm I}}_{[0;v_k]},{1\!\!\,{\rm I}}_{[0;1]})}},
\end{equation}
where $c(\vec{s}, k)$ is a positive constant which depends on $s_1, s_2, k$ and $\widetilde{P}$ is a projection onto linear span generated by ${1\!\!\,{\rm I}}_{[0;1]}.$ One can check that \eqref{eq23} equals
$$
\widetilde{c}(\vec{s}, k)E(l^{\widetilde{w}})^k,
$$
where $\widetilde{w}(t)=w(t)-tw(1),\ t\in[0;1]$ is the Brownian bridge.
To finish the proof it suffices to check that
$$
\sup_{k\geq1}E(l^{\widetilde{w}})^k<+\infty.
$$
Really, note that
$$
E(l^{\widetilde{w}})^k=EE((l^{w})^k/w(1)=0).
$$
It is known that joint probability density of the random variables $l^w$ and $w(1)$ has the following representation \cite{12}
$$
p(a,b)=\frac{1}{\sqrt{2\pi}}(|b|+a)e^{-\frac{1}{2}(|b|+a)^2},\ a>0,\  b\in{\mathbb R}.
$$
Then
$$
E((l^{w})^k/w(1)=0)=\frac{\int^{+\infty}_0y^kp(y,0)dy}
{\int^{+\infty}_0p(y,0)dy}=
$$
$$
=\frac{\int^{+\infty}_0y^{k+1}e^{-\frac{y^2}{2}}dy}
{\int^{+\infty}_0ye^{-\frac{y^2}{2}}dy}
=2^{\frac{k}{2}}\Gamma(\frac{k}{2}+1),
$$
where
$$
\Gamma(z)=\int^{+\infty}_0t^{z-1}e^{-t}dt,\ z\in{\mathop{\rm Re}},\ {\mathop{\rm Re}}\,z>0
$$
is the gamma function.
\end{proof}

 \end{document}